\documentclass[12pt]{amsart}
\newtheorem{theorem}{Theorem}[section]
\newtheorem{lemma}[theorem]{Lemma}

\newtheorem{proposition}[theorem]{Proposition}

\newtheorem{example}[theorem]{Example}

\def\erre{{\rm I\!R}}
\def\RR{{\rm I\!R}}
\def\R{{\rm I\!R}}
\def\enne{{\rm I\!N}}
\def\meas{\mathop{\rm meas}}

\def\R{{\rm I\!R}}
\def\R{{\rm I\!R}}
\def\N{{\rm I\!N}}

\def\phi{\varphi}

\def\meas{\mathop{\rm meas}}

\title[Sequences of weak solutions for...]{Sequences of weak solutions for fractional equations}
\author{Giovanni Molica Bisci}
\address[G. Molica Bisci]{Dipartimento P.A.U., Universit\`a  degli
Studi ``Mediterranea" di Reggio Calabria, Salita Melissari - Feo di
Vito, 89100 Reggio Calabria, Italy} \email{gmolica@unirc.it}

\thanks{{The paper is realized with the auspices of the GNAMPA Project 2013 entitled: {\it Problemi non-locali di tipo Laplaciano frazionario}}}
\begin{document}

\begin{abstract}
This work is devoted to study the existence of infinitely many weak solutions to nonlocal equations involving a general
integrodifferential operator of fractional type.
These equations have a
variational structure and we find a sequence of nontrivial weak solutions for them
exploiting the ${\mathbb{Z}}_2$-symmetric version of the Mountain Pass Theorem.  To make the nonlinear methods work, some
careful analysis of the fractional spaces involved is necessary. As a particular case, we derive an
existence theorem for the fractional Laplacian, finding nontrivial
solutions of the equation
$$ \left\{
\begin{array}{ll}
(-\Delta)^s u=f(x,u) & {\mbox{ in }} \Omega\\
u=0 & {\mbox{ in }} \erre^n\setminus \Omega.
\end{array} \right.
$$
As far as we know, all these results are new and represent a fractional version of classical theorems obtained working with Laplacian equations.
\end{abstract}
\maketitle
\section{Introduction}
In this paper we obtain an existence result for the following nonlocal problem:
\begin{equation} \tag{$P_{K}^{f}$} \label{Nostro}
\left\{\begin{array}{ll}
-\mathcal L_K u=f(x, u) & \mbox{in}\,\,\,\Omega\\
u=0 & \mbox{in}\,\,\,\erre^n\setminus\Omega.
\end{array}\right.
\end{equation}

Here and in the sequel, $\Omega$ is a bounded domain in $({\R}^{n},|\cdot|)$ with $n>2s$ (where $s\in (0,1)$) and with smooth (Lipschitz)
boundary $\partial \Omega$, $f:\bar\Omega\times\erre\rightarrow \erre$ is a suitable continuous function with subcritical growth and $\mathcal L_K$ is the nonlocal operator defined as follows:
\begin{equation*}
\mathcal L_Ku(x):=
\int_{\erre^n}\Big(u(x+y)+u(x-y)-2u(x)\Big)K(y)dy,
\,\,\, (x\in \erre^n)
\end{equation*}
where $K:\erre^n\setminus\{0\}\rightarrow(0,+\infty)$ is a function with the properties that:
\begin{itemize}
\item[$(\rm k_1)$] $\gamma K\in L^1(\erre^n)$, \textit{where} $\gamma(x)=\min \{|x|^2, 1\}$;
\item[$(\rm k_2)$] \textit{There exists} $\lambda>0$
\textit{such that} $$K(x)\geq \lambda |x|^{-(n+2s)},$$
{\textit{for any}} $x\in \erre^n \setminus\{0\}$;
\item[$(\rm k_3)$] $K(x)=K(-x)$, \textit{for any} $x\in \erre^n \setminus\{0\}$.
\end{itemize}

\indent A typical example for the kernel $K$ is given by $K(x):=|x|^{-(n+2s)}$. In this case $\mathcal L_K$ is the fractional Laplace operator defined as
$$
-(-\Delta)^s u(x):=
\int_{\erre^n}\frac{u(x+y)+u(x-y)-2u(x)}{|y|^{n+2s}}\,dy,
\,\,\,\,\, x\in \erre^n.
$$

\indent Recently, a great attention has been focused on the study of
fractional and nonlocal operators of elliptic type, both for the
pure mathematical research and in view of concrete real-world
applications. This type of operators arises in a quite natural way
in many different contexts, such as, among the others,
the thin obstacle problem,
optimization,
finance, phase transitions,
stratified materials,
anomalous diffusion, crystal dislocation,
soft thin films,
semipermeable membranes, flame propagation,
conservation laws,
ultra-relativistic
limits of quantum mechanics,
quasi-geostrophic flows,
multiple scattering,
minimal surfaces,
materials science and
water waves.\par
\indent In this paper, problem \eqref{Nostro} is studied exploiting classical variational methods. More precisely we apply the ${\mathbb{Z}}_2$-symmetric version of the Mountain Pass Theorem (briefly ${\mathbb{Z}}_2$-MPT) to this kind of equations motivated by the current literature where the MPT has been intensively applied to find solutions to
quasilinear elliptic equations; see \cite{ar,puccirad, rabinowitz,struwe}.\par
\indent Technically, this approach is realizable checking that the associated energy functional verifies the usual compactness Palais-Smale condition in a suitable variational setting developed in \cite{svmountain}. Indeed, the nonlocal analysis that we perform here
in order to use Mountain Pass Theorem is quite general and successfully exploited for other goals
in several recent contributions; see \cite{svmountain,svlinking,servadeivaldinociBN} and \cite{valpal}
for an elementary introduction to this topic and for a list of related references.\par
This functional analytical context is inspired by (but not equivalent
to) the fractional Sobolev spaces, in order to correctly encode
the Dirichlet boundary datum in the variational formulation.\par
Further, we suppose that the right-hand side of
equation~\eqref{Nostro} is a continuous odd function
$f:\bar\Omega\times \erre^n\to \erre$ verifying the following conditions:

\begin{itemize}
\item[$(\rm h_1)$] \textit{There exist $a_1, a_2>0$ and $q\in (2, 2^*)$, $2^*:=\displaystyle\frac{2n}{n-2s}$ such that}
$$
|f(x,t)|\le a_1+a_2|t|^{q-1},
$$
\textit{for every $x\in \bar\Omega$, $t\in \erre$};
\item[$(\rm h_2)$] \textit{There are two constants $\theta>2$ and $r>0$ such that}
$$
0<\theta F(x,t)\le tf(x,t),
$$
\textit{for every $x\in \bar\Omega$, and $|t|\geq r$},
\end{itemize}
where the function $F$ is
the primitive of $f$ with respect to the second variable, that is
\begin{equation}\label{F}
{\displaystyle F(x,t):=\int_0^t f(x,s)ds},\quad\, (\forall\, t\in\erre).
\end{equation}
\indent Under the previous assumptions, we prove the existence
of infinitely many weak solutions to problem \eqref{Nostro}; see Theorem \ref{Esistenza}. Note that the symmetry hypothesis on $f$ allows to remove any condition at zero.\par
In the nonlocal framework, the simplest example we can deal with is
given by the fractional Laplacian, according to the following
proposition.

\begin{theorem}\label{lapfra0}
Let $s\in (0,1)$, $n>2s$ and $\Omega$ be an open bounded set of $\erre^n$
with
Lipschitz boundary.\par
Consider the following equation
\begin{equation}\label{problemalapfrac0}
{\displaystyle \int_{\erre^{n}\times\erre^n}
\frac{(u(x)-u(y))(\varphi(x)-\varphi(y))}{|x-y|^{n+2s}}dxdy\!=\!\!\int_\Omega
f(x,u(x))\,\varphi(x)dx}
\end{equation}
for any $\varphi\in H^s(\erre^n)$ with $\varphi=0$ a.e. in $\erre^n\setminus\Omega$.\par
 If $f:\bar\Omega\times \erre \to \erre$ is an odd continuous function verifying $(\rm h_1)$ and $(\rm h_2)$, then problem~\eqref{problemalapfrac0} admits a sequence of infinitely many solutions $\{u_j\}\subset H^s(\erre^n)$,
such that $u_j=0$ a.e. in $\erre^n\setminus\Omega$.
\end{theorem}
 The above result is the fractional analogous of \cite[Theorem 9.38]{rabinowitz} in which the classical Dirichlet problem
\begin{equation} \tag{$D_{f}$} \label{SV}
\left\{
\begin{array}{ll}
-\Delta u
= f(x,u) & \rm in \quad \Omega
\\u=0 & \rm on\,\,\partial\Omega,\\
\end{array}
\right.
\end{equation}
is studied; see \cite[Theorem 6.6]{struwe}. We also cite \cite{ar,puccirad} where MPT (and some of its variant) has been intensively applied to find solutions to quasilinear elliptic equations.\par
\indent The plan of the paper is as follows; Section 2 is devoted to our abstract framework and preliminaries. Successively, in Sections 3 we give the main result; see Theorem \ref{Esistenza}. Finally, a concrete example of application is presented in Example \ref{esempio0}.
\section{Abstract Framework}\label{section2}
In this subsection we briefly recall the definition of the functional space~$X_0$, firstly introduced in \cite{sv}, and we give some notations.
The reader familiar with this topic may skip this section and go directly to the next one.

The functional space $X$ denotes the linear space of Lebesgue
measurable functions from $\RR^n$ to $\RR$ such that the restriction
to $\Omega$ of any function $g$ in $X$ belongs to $L^2(\Omega)$ and
$$
((x,y)\mapsto (g(x)-g(y))\sqrt{K(x-y)})\in
L^2\big((\RR^n\times\RR^n) \setminus ({\mathcal C}\Omega\times
{\mathcal C}\Omega), dxdy\big),$$ where ${\mathcal C}\Omega:=\RR^n
\setminus\Omega$.\par
 We denote by $X_0$ the following linear
subspace of $X$
$$X_0:=\big\{g\in X : g=0\,\, \mbox{a.e. in}\,\,
\RR^n\setminus
\Omega\big\}.$$
We remark that $X$ and $X_0$ are non-empty, since $C^2_0 (\Omega)\subseteq X_0$ by \cite[Lemma~11]{sv}.
Moreover, the space $X$ is endowed with the norm defined as
$$
\|g\|_X:=\|g\|_{L^2(\Omega)}+\Big(\int_Q |g(x)-g(y)|^2K(x-y)dxdy\Big)^{1/2}\,,
$$
where $Q:=(\RR^n\times\RR^n)\setminus \mathcal O$ and
${\mathcal{O}}:=({\mathcal{C}}\Omega)\times({\mathcal{C}}\Omega)\subset\RR^n\times\RR^n$.\par
It is easily seen that $\|\cdot\|_X$ is a norm on $X$; see, for
instance, \cite{svmountain} for a proof.\par

By \cite[Lemmas~6 and 7]{svmountain} in the sequel we can take the function
\begin{equation}\label{normaX0}
X_0\ni v\mapsto \|v\|_{X_0}:=\left(\int_Q|v(x)-v(y)|^2K(x-y)dxdy\right)^{1/2}
\end{equation}
as norm on $X_0$.\par
 Also $\left(X_0, \|\cdot\|_{X_0}\right)$ is a Hilbert space with scalar product
$$
\langle u,v\rangle_{X_0}:=\int_Q
\big( u(x)-u(y)\big) \big( v(x)-v(y)\big)\,K(x-y)dxdy,
$$
see \cite[Lemmas~7]{svmountain}.\par
Note that in \eqref{normaX0} (and in the related scalar product) the integral can be extended to all $\RR^n\times\RR^n$, since $v\in X_0$ (and so $v=0$ a.e. in $\RR^n\setminus \Omega$).

In what follows, we denote by $\lambda_k$ be the $k$-th eigenvalue of
the operator~$-\mathcal L_K$ with homogeneous Dirichlet boundary
data, namely the $k$-th eigenvalue of the problem
\begin{equation}\label{problema autovalori}
\left\{\begin{array}{ll}
-\mathcal L_K u=\lambda u & \mbox{in } \Omega\\
u=0 & \mbox{in } \mathbb{R}^{n}\setminus\Omega.
\end{array}
\right.
\end{equation}

\indent The set of the eigenvalues of problem \eqref{problema autovalori} consists of a sequence $\left\{\lambda_k\right\}_{k\in\mathbb{N}}$ with\footnote{As usual, here we call $\lambda_1$ the \emph{first
eigenvalue}
of the operator $-\mathcal L_K$. This notation is
justified by~\eqref{ordine lambdak}. Notice also that some of the eigenvalues
in the sequence $\big\{ \lambda_k\big\}_{{k\in\enne}}$
may repeat, i.e. the inequalities in~\eqref{ordine lambdak} may
be not always strict.}
\begin{equation}
0<\lambda_1<\lambda_2\leq\ldots\leq\lambda_k\leq\lambda_{k+1}\leq\ldots\label{ordine lambdak}
\end{equation}
and
\begin{equation*}
\lambda_k\rightarrow +\infty\,\,\mbox{as }\,\,k\rightarrow \infty.
\end{equation*}

\indent Further, the following characterization holds:
$$
\lambda_{k}=\min_{u\in \mathbb P_{k}\setminus{\left\{0\right\}}}\frac{\displaystyle \int_{\mathbb{R}^{2n}}\left|u(x)-u(y)\right|^{2}K(x-y)dx\,dy}
{\displaystyle\int_{\Omega}u(x)^{2}dx},\label{lambda_k}
$$
where
\begin{equation}\label{pk+1}
 \mathbb P_{k}:=\left\{u\in X_0:\,\left\langle u,e_j\right\rangle_{X_0}=0\quad\forall j=1,\ldots,k\right\}.
\end{equation}

Finally, the symbol $E(\lambda_k)$ denotes the linear space of eigenfunctions corresponding to $\lambda_k$. For the existence and the basic properties of this eigenvalue we
refer to \cite[Proposition~9 and Appendix~A]{svlinking}, where a
spectral theory for these general integrodifferential nonlocal operators
was developed. Further properties can be also found in \cite{sY}.\par

While for a general kernel~$K$ satisfying
conditions from $(\rm k_1)$ to $(\rm k_3)$ we have that
$X_0\subset H^s(\RR^n)$, in the model case $K(x):=|x|^{-(n+2s)}$ the
space $X_0$ consists of all the functions of the usual fractional
Sobolev space $H^s(\RR^n)$ which vanish a.e. outside $\Omega$; see
\cite[Lemma~7]{servadeivaldinociBN}.\par
 Here $H^s(\RR^n)$ denotes
the usual fractional Sobolev space endowed with the norm (the
so-called \emph{Gagliardo norm})
$$
\|g\|_{H^s(\RR^n)}=\|g\|_{L^2(\RR^n)}+
\Big(\int_{\RR^n\times\RR^n}\frac{\,\,\,|g(x)-g(y)|^2}{|x-y|^{n+2s}}\,dxdy\Big)^{1/2}.
$$
Before concluding this subsection, we recall the embedding
properties of~$X_0$ into the usual Lebesgue spaces; see
\cite[Lemma~8]{svmountain}.\par
 The embedding $j:X_0\hookrightarrow
L^{\nu}(\RR^n)$ is continuous for any $\nu\in [1,2^*]$, while it is
compact whenever $\nu\in [1,2^*)$, where $2^*:=2n/(n-2s)$. Hence, for any $\nu\in [1,2^*]$
there exists a positive constant $c_\nu$ such that
\begin{equation}\label{Poincare1}
\|v\|_{L^{\nu}(\erre^n)}\leq c_{\nu} \|v\|_{X_0} \quad \mbox{for any}\,\,v\in
X_0.
\end{equation}

For further details on the fractional Sobolev spaces we refer
to~\cite{valpal} and to the references therein, while for other
details on $X$ and $X_0$ we refer to \cite{sv}, where these
functional spaces were introduced, and also to \cite{sY, svmountain,
svlinking, servadeivaldinociBN}, where various properties of these spaces were
proved.

Finally, for the sake of completeness, we recall that a $C^1$-functional $J:E\to\R$, where $E$ is a real Banach
space with topological dual $E^*$, satisfies the \textit{Palais-Smale condition at level} $\mu\in\R$, (briefly $\textrm{(PS)}_{\mu}$) when:
\begin{itemize}
\item[$\textrm{(PS)}_{\mu}$] {\it Every sequence $\{u_j\}_{j\in\mathbb{N}}$ in $E$ such that
$$
J(u_j)\to \mu, \quad{\rm and}\quad \|J'(u_j)\|_{E^*}\to0,
$$
as $j\rightarrow \infty$, possesses a convergent subsequence.}
\end{itemize}
We say that $J$ satisfies the \textit{Palais-Smale condition} (in short $\textrm{(PS)}$) if $\textrm{(PS)}_{\mu}$ holds for every $\mu\in \erre$.\par
With the above notation, our main tool is the following classical result recalled in a convenient form.

\begin{theorem}\label{ZMPT}
Let $E$ be an infinite dimensional real Banach space and let $J\in C^1(E;\erre)$ be even, satisfying the $(\rm PS)$ condition and $J(0_E)=0$. Suppose $E=E_1\oplus E_2$, where $E_1$ is finite dimensional and $J$ satisfies$:$
 \begin{itemize}
  \item[$(I_1)$] There exist constant $\rho,\alpha>0$ such that $$J(u)\geq \alpha,$$ for every $u\in E_2$ and $\|u\|_E=\rho$.
  \item[$(I_2)$] For each finite dimensional subspace $W\subset E$, the set
  $$
  \{u\in W:J(u)\geq 0\}
  $$
  is bounded in $E$.
 \end{itemize}
 Then $J$ has an unbounded sequence of critical values.
\end{theorem}
See \cite[Theorem 9.12]{rabinowitz}.\par
\smallskip
 \noindent We cite the monograph \cite{k2} as general reference on variational methods adopted in this paper.
\section{The main Theorem}
Our result is as follows.
\begin{theorem}\label{Esistenza}
Let $s\in (0,1)$, $n>2s$ and $\Omega$ be an open bounded set of
$\RR^n$ with Lipschitz boundary and
$K:\RR^n\setminus\{0\}\rightarrow(0,+\infty)$ be a map
satisfying $(\rm k_1)$--$(\rm k_3)$.\par
 In addition, let $f:\bar\Omega\times\erre\to\erre$ be an odd
continuous function verifying $(\rm h_1)$ and $(\rm h_2)$.\par
 Then, problem~\eqref{Nostro} possesses an unbounded sequence of weak solutions.
\end{theorem}

We recall that a \textit{weak solution} of problem \eqref{Nostro} is a function $u\in X_0$ such that
$$
\begin{array}{l} {\displaystyle \int_Q \big(u(x)-u(y)\big)\big(\varphi(x)-\varphi(y)\big)K(x-y)dxdy}\\
\displaystyle \qquad\qquad\qquad\qquad\qquad\quad\quad= {\int_\Omega f(x,u(x))\varphi(x)dx},\,\,\,\,\,\,\, \forall\,\, \varphi \in X_0.
\end{array}
$$

 \indent For the proof of Theorem~\ref{Esistenza}, we observe that
problem~\eqref{Nostro} has a variational structure, indeed it is the Euler-Lagrange equation of the functional $J_K:X_0\to \RR$ defined as follows
\begin{equation}\label{functional}
J_K(u):=\frac 1 2 \int_Q|u(x)-u(y)|^2 K(x-y)dxdy-\int_\Omega F(x, u(x))dx.
\end{equation}
Note that the functional $J_K$ is Fr\'echet differentiable in $u\in X_0$ and for any $\varphi\in X_0$ one has
$$\begin{aligned}
\langle J'_K(u), \varphi\rangle & = \int_Q \big(u(x)-u(y)\big)\big(\varphi(x)-\varphi(y)\big)K(x-y)dxdy\\
& \qquad \qquad \qquad \qquad \qquad \qquad\qquad -\int_\Omega f(x, u(x))\varphi(x)dx.
\end{aligned}$$

\indent Thus, critical points of $J_K$ are solutions to problem~\eqref{Nostro}.
In order to find these critical points, we will make use of Theorem \ref{ZMPT}.
For this, we have to check that the functional $J_K$ has a particular geometric
structure and satisfies the Palais--Smale compactness condition.

\subsection{Proof of Theorem \ref{Esistenza}}
 In order to prove our result, we apply (as claimed before) Theorem \ref{ZMPT} to the functional $J_K$ defined in \eqref{functional}. The conclusion of Theorem \ref{Esistenza} is equivalent to the assertion that $J_K$ admits an unbounded sequence of critical points.
 Hence, let us start observing that, since $f$ is odd in the second variable i.e. $f(x,-t)=-f(x,t)$, for every $t\in\erre$, $J_K$ is even.  Moreover, by definition, $J_K(0_{X_0})=0$.\par
  In the next two lemmas
 we shall verify the compactness $(\rm PS)$ condition.
\begin{lemma}\label{lemma1}
 Every Palais-Smale sequence for the functional $J_K$
 is bounded in $X_0$.
\end{lemma}
\begin{proof}
Let $\{u_j\}_{j\in\mathbb{N}}\subset X_0$ be a Palais-Smale sequence i.e.
\begin{equation}\label{E1}
J_K(u_j)\rightarrow \mu,
\end{equation}
for $\mu\in \erre$ and
\begin{equation}\label{E2}
\quad \|J'_K(u_j)\|_{X_0^*}=\sup\Big\{ \big|\langle\,J'_K(u_j),\varphi\,\rangle \big|\,: \;
\varphi\in
X_0\,, \|\varphi\|_{X_0}=1\Big\}\to0,
\end{equation}
as $j\rightarrow \infty$.\par
\indent We argue by contradiction. So, suppose that the conclusion is not true. Passing to a subsequence if necessary, we may assume that
$$
\|u_j\|_{X_0}\rightarrow +\infty,
$$
as $j\rightarrow \infty$.\par
\indent  By definition it follows that
  \begin{eqnarray*}
J_K(u_j) &-& \frac{\langle J'_K(u_j),u_j\rangle }{\theta}\\\nonumber
               &=& \left(\frac{1}{2}-\frac{1}{\theta}\right)\int_Q|u_j(x)-u_j(y)|^2K(x-y)dxdy\\\nonumber
               &+& \int_{\Omega}\left[\frac{f(x,u_j(x))u_j(x)}{\theta}-F(x,u_j(x))\right]dx,\nonumber
\end{eqnarray*}
for every $j\in \enne$.\par
 Thus
\begin{eqnarray*}
\left(\frac{\theta-2}{2\theta}\right)\|u_j\|_{X_0}^2&\leq& J_K(u_j)-\frac{\langle J'_K(u_j),u_j\rangle }{\theta}\\\nonumber
               &-& \int_{|u_j(x)|>r}\left[\frac{f(x,u_j(x))u_j(x)}{\theta}-F(x,u_j(x))\right]dx,\\\nonumber
               &+& M \meas(\Omega),\quad\, \forall\, j\in \enne,
\end{eqnarray*}
where $``\meas(\Omega)"$ denotes the standard Lebesgue measure of $\Omega$ and
$$
M:=\max\left\{\left|\frac{f(x,t)t}{\theta}-F(x,t)\right|:x\in\bar\Omega, |t|\leq r\right\}.
$$

\indent Now, we observe that, the Ambrosetti Rabinowitz condition yields
$$
\int_{|u_j(x)|>r}\left[\frac{f(x,u_j(x))u_j(x)}{\theta}-F(x,u_j(x))\right]dx\geq 0.
$$
So, we deduce that
\begin{eqnarray*}
\left(\frac{\theta-2}{2\theta}\right)\|u_j\|_{X_0}^2\leq J_K(u_j)-\frac{\langle J'_K(u_j),u_j\rangle }{\theta}+M \meas(\Omega),
\end{eqnarray*}
for every $j\in \enne$.\par
Then, for every $j\in \enne$ one has
\begin{eqnarray*}
C\|u_j\|_{X_0}^2\leq J_K(u_j)+{\theta}{\|J'_K(u_j)\|_{X_0^*}\|u_j\|_{X_0}}+M \meas(\Omega),
\end{eqnarray*}
where $C:=\displaystyle m_0\left(\frac{\theta-2}{2\theta}\right)>0$.\par
In conclusion, dividing by $\|u_j\|_{X_0}$ and letting $j\rightarrow \infty$, we obtain a contradiction.
\end{proof}

The above Lemma implies that the $C^1$-functional
$J_K$ verifies the Palais-Smale condition as proved in the next result.
\begin{lemma}\label{lemma2}
The functional $J_K$
satisfies the compactness $(\rm PS)$ condition.
\end{lemma}

\begin{proof}

Let $\{u_{j}\}_{j\in\mathbb{N}}\subset X_0$ be a Palais-Smale
sequence and,
in order to simplifies the notations, let us put
$$
\Phi(u):=\frac 1 2 \int_Q|u(x)-u(y)|^2 K(x-y)dxdy,
$$
for every $u\in X_0$.\par
Taking into account Lemma \ref{lemma1}, the sequence $\{u_j\}_{j\in\mathbb{N}}$ is necessarily bounded in $X_0$. Since $X_0$ is reflexive, we may extract a subsequence
that for simplicity we call again $\{u_{j}\}_{j\in\mathbb{N}}$, such that
$u_{j}\rightharpoonup u_\infty$ in $X_0$. This means that
\begin{equation}\label{convergenze0}
\begin{aligned}
 & \int_Q
\big( u_j(x)-u_j(y)\big) \big(\phi(x)-\phi(y)\big)\,K(x-y)dxdy \to \\
& \qquad \qquad \qquad
\int_Q
\big( u_\infty(x)-u_\infty(y)\big) \big(\phi(x)-\phi(y)\big)\,K(x-y)dxdy ,
\end{aligned}
\end{equation}
for any $\phi\in X_0$,
as $j\to \infty$.\par
 We will prove that ${u_{j}}$ strongly converges to $u_\infty\in X_0$.
Exploiting the derivative $J'_K(u_j)(u_j-u_\infty)$, we obtain
\begin{eqnarray}\label{jj}
\langle \Phi'(u_{j}),u_{j}-u_\infty\rangle &=& \langle J'_K(u_j),u_j-u_\infty\rangle\\
               &+ & \int_{\Omega}f(x,u_j(x))(u_j-u_\infty)(x)dx,\nonumber
\end{eqnarray}
where
\begin{eqnarray*}
\langle \Phi'(u_{j}),u_{j}-u_\infty\rangle  &=& \int_Q|u_j(x)-u_j(y)|^2K(x-y)dxdy\\\nonumber
               &-& \int_Q
\big( u_j(x)-u_j(y)\big) \big(u_\infty(x)-u_\infty(y)\big)\,K(x-y)dxdy\nonumber
\end{eqnarray*}
\indent Since $\|J'_K(u_j)\|_{X_0^{*}}\to0$ and the sequence $\{u_j-u_\infty\}$ is bounded in
$X_0$, taking into account that $|\langle
J'_K(u_j),u_j-u_\infty\rangle|\leq\|J'_K(u_j)\|_{X_0^{*}}\|u_j-u_\infty\|_{X_0}$, one has
\begin{eqnarray}\label{j2}
\langle J'_K(u_j),u_j-u_\infty\rangle\to0,
\end{eqnarray}
as $j\rightarrow \infty$.\par
 At this point, wee observe that, since the embedding
$X_0\hookrightarrow L^q(\Omega)$ is compact, clearly $u_j\to u_\infty$ strongly in $L^q(\Omega)$.
So by condition $(\textrm{h}_1)$, we easily obtain that
\begin{eqnarray}\label{j3}
\int_{\Omega}|f(x,u_j(x))||u_j(x)-u_\infty(x)|dx\to0,
\end{eqnarray}
as $j\rightarrow \infty$.\par
 \indent By \eqref{jj} relations \eqref{j2} and \eqref{j3} yield
\begin{eqnarray}\label{fin}
\langle \Phi'(u_{j}),u_{j}-u_\infty\rangle
\rightarrow 0,
\end{eqnarray}
as $j\rightarrow \infty$.\par
Hence by \eqref{fin} we can write
\begin{eqnarray}\label{fin2}
&& \int_Q|u_j(x)-u_j(y)|^2K(x-y)dxdy-\\\nonumber
               && \quad\quad\int_Q
\big( u_j(x)-u_j(y)\big) \big(u_\infty(x)-u_\infty(y)\big)\,K(x-y)dxdy\rightarrow 0,\nonumber
\end{eqnarray}
as $j\rightarrow \infty$.\par
Thus, by \eqref{fin2} and \eqref{convergenze0} it follows that
\begin{eqnarray*}
&& \lim_{j\rightarrow \infty}\int_Q|u_j(x)-u_j(y)|^2K(x-y)dxdy\\\nonumber
               && \quad\quad\quad\quad\quad\quad=\int_{Q}\left|u_\infty(x)-u_\infty(y)\right|^{2}K(x-y)dxdy.\nonumber
\end{eqnarray*}
\indent In conclusion, thanks to \cite[Proposition III.30]{brezis}, $u_j\rightarrow u_\infty$ in $X_0$. The proof is complete.
\end{proof}

The proof of the main result is concluded if we prove that hypotheses $(I_1)$ and $(I_2)$ in Theorem \ref{ZMPT} are verified. We show these facts in the next two propositions.

\begin{proposition}\label{proposition1}
The functional $J_K$ satisfies condition $(I_1)$.
\end{proposition}
\begin{proof}

We claim that there exists $k_0\in\enne$ sufficiently large and two positive constants $\rho$ and $\alpha$ such that for every
$$u\in E_2:=\displaystyle\overline{\bigoplus_{k\geq k_0}E(\lambda_k)}^{\|\cdot\|_{X_0}},$$ with $\|u\|_{X_0}=\rho,$ there holds $J_K(u)\geq \alpha$.\par
 Indeed, since $q\in (2,2^*)$ one has
\begin{eqnarray}\label{interpolation}
\|u\|_{L^q(\Omega)}^{q}\leq \|u\|_{L^2(\Omega)}^{\beta}\|u\|_{L^{2^*}(\Omega)}^{q-\beta},\qquad (\forall\, u\in X_0)
\end{eqnarray}
where we set $\displaystyle \beta:=2\Big(\frac{2^*-q}{2^*-2}\Big)$; see, for instance, \cite[p. 105]{adams}.\par

Now, let us denote by $\left\{e_k\right\}_{k\in\N}$ the sequence of eigenfunctions.  By \cite[Proposition~9 and Appendix A]{svlinking}, we have that the sequence $\left\{e_k\right\}_{k\in\N}$ is an orthonormal basis of $L^{2}(\Omega)$ and an orthogonal basis of $X_0$.\par
\indent More precisely, for every $k\in\enne$ one has
 \begin{equation}\label{p1}
\langle e_k,e_k\rangle_{X_0}=\lambda_k\int_{\Omega}e_k(x)^2dx=\lambda_k,
\end{equation}
\noindent and
 \begin{equation}\label{p2}
\langle e_i,e_j\rangle_{X_0}=\int_Q
\big( e_i(x)-e_i(y)\big) \big(e_j(x)-e_j(y)\big)\,K(x-y)\,dx\,dy=\delta_{i}^{j},
\end{equation}
 \noindent for all $i,j\in \N$.\par
  If $u\in E_2$, $\displaystyle u=\sum_{j=k_0}^{+\infty}\beta_je_j$,
for suitable $\beta_j\in\R$, where $j\in \N$ and $j\geq k_0$.
  Hence, by using \eqref{p1} and \eqref{p2} one has
$$
\|u\|^2_{L^2(\Omega)}=\sum_{j=k_0}^{+\infty}\beta_j^2\int_{\Omega} e_j(x)^2
dx=\sum_{j=k_0}^{+\infty}\frac{\beta_j^2}{\lambda_j}\langle e_j,e_j\rangle_{X_0}\leq \frac{1}{\lambda_{k_0}}\|u\|^2_{X_0},
$$
\noindent i.e.,
\begin{equation}\label{i3}
\|u\|_{L^2(\Omega)}\leq \frac{1}{\sqrt{\lambda_{k_0}}}\|u\|_{X_0},
\end{equation}
\noindent for each $u\in E_2$.\par
By the growth condition $(\rm h_1)$ there is a positive constants $C_1$ such that
$$
|F(x,\xi)|\le C_1(1+|t|^{q}),
$$
{for every $x\in \bar\Omega$, $\xi\in \erre$}.\par
Then, by \eqref{interpolation}, \eqref{i3} and using the Sobolev's embedding $X_0\hookrightarrow L^{2^*}(\Omega)$ we have
\begin{eqnarray*}
J_K(u) &\geq& \displaystyle\frac{\|u\|^{2}_{X_0}}{2}-C_1\|u\|_{L^q(\Omega)}^q-C_1\meas(\Omega)\\\nonumber
&\geq & \displaystyle\frac{\|u\|^{2}_{X_0}}{2}-C_1\|u\|_{L^2(\Omega)}^\beta\|u\|_{L^{2^*}(\Omega)}^{q-\beta}-C_1\meas(\Omega)\\\nonumber
               &\geq& \displaystyle\frac{\|u\|^{2}_{X_0}}{2}-C_1\frac{c_{2^*}^{q-\beta}}{\lambda_{k_0}^{\beta/2}}\|u\|_{X_0}^q-C_1\meas(\Omega)\\\nonumber
               &\geq& \left(\displaystyle\frac{1}{2}-C_1\frac{c_{2^*}^{q-\beta}}{\lambda_{k_0}^{\beta/2}}\|u\|_{X_0}^{q-2}\right)\|u\|_{X_0}^2-C_1\meas(\Omega),\nonumber
\end{eqnarray*}
for every $u\in E_2$.\par
We may let $\rho:=2\sqrt{C_1\meas(\Omega)+1}$ and choose $k_0\in\enne$ such that
$$
\lambda_{k_0}\geq\left({2^qC_1c_{2^*}^{q-\beta}(C_1\meas(\Omega)+1)^{(q-2)/2}}\right)^{2/\beta},
$$
to achieve that
$$
J_K(u)\geq 1,\quad (\alpha:=1)
$$
for every $u\in E_2$ and $\|u\|_E=\rho$.\par
 Hence, condition $(I_1)$ is verified. The proof is complete.
\end{proof}

\begin{proposition}\label{proposition2}
The functional $J_K$ satisfies condition $(I_2)$.
\end{proposition}
\begin{proof}
Let $W\subset X_0$ be a finite dimensional space. We prove the set
  $$
  \{u\in W: J_K(u)\geq 0\}
  $$
  is bounded in $X_0$.\par
  Indeed, let $u\in X_0$ arbitrary and denote
  $$
  \Omega_{<}:=\{x\in \Omega:|u(x)|<r\},
  $$
  as well as
  $$
  \Omega_{\geq}:=\{x\in \Omega:|u(x)|\geq r\}.
  $$

  \indent We shall prove that $J_K$ satisfies the following estimate
  \begin{equation}\label{15}
J_K(u) \leq \frac{\|u\|^{2}_{X_0}}{2} -\int_\Omega \gamma(x)|u(x)|^{\theta}dx+\kappa,
\end{equation}
  where $\kappa$ is a suitable positive constant and $\gamma\in L^{\infty}(\Omega)$, with $\gamma>0$ in $\Omega$.\par
\indent To show the above inequality, let us start observing that by $(\rm h_1)$, as pointed out before, the function $F$ satisfies
  \begin{equation}\label{13}
 |F(x,\xi)|\leq C_1(1+|\xi|^{q}),\qquad\forall\, x\in \bar\Omega,\,\forall\,\xi\in \erre.
\end{equation}

\indent We claim that there exists $\gamma\in L^{\infty}(\Omega)$, $\gamma>0$ in $\Omega$, such that
\begin{equation}\label{14}
F(x,\xi)\geq \gamma(x)|\xi|^\theta,\qquad\forall\, x\in \bar\Omega,\,\forall\,|\xi|\geq r.
\end{equation}
Indeed, since $F$ is $\theta$-superhomogeneous, we have that
$$
F(x,\xi)\geq \gamma_1(x)|\xi|^\theta,\qquad\forall\, x\in \bar\Omega,\,\forall\,\xi\geq r,
$$
where $\displaystyle \gamma_1(x):=F(x,r)/r^\theta$. It is easy to see that $\gamma_1\in L^\infty(\Omega)$ and $\gamma_1>0$ in $\Omega$.\par
 In a similar way, it follows that
$$
F(x,\xi)\geq \gamma_2(x)|\xi|^\theta,\qquad\forall\, x\in \bar\Omega,\,\forall\,\xi\leq -r,
$$
with $\displaystyle \gamma_2(x):=F(x,-r)/r^\theta$. Also in this case $\gamma_2\in L^\infty(\Omega)$ and $\gamma_2>0$ in $\Omega$.\par
\indent Then \eqref{14} holds with
$$
\gamma(x):=\min\{\gamma_1(x),\gamma_2(x)\},
$$
for every $x\in \bar\Omega$.\par
Now, by condition \eqref{13} we conclude that
$$
\int_{\Omega_{<}}F(x,u(x))dx\geq -C_1(r^q+1)\meas\,(\Omega).
$$
Further, inequality \eqref{14} yields
$$
\int_{\Omega_{\geq}}F(x,u(x))dx\geq \int_{\Omega_{\geq}}\gamma(x)|u(x)|^{\theta}dx.
$$
\indent Then
\begin{eqnarray*}
J_K(u) &\leq& \frac{\|u\|^{2}_{X_0}}{2}-\left(\int_{\Omega_{<}}F(x,u(x))dx+\int_{\Omega_{\geq}}F(x,u(x))dx\right)\\\nonumber
&\leq& \frac{\|u\|^{2}_{X_0}}{2}-\int_{\Omega_{\geq}}\gamma(x)|u(x)|^{\theta}dx+C_1(r^q+1)\meas\,(\Omega)\\\nonumber
               &\leq& \frac{\|u\|^{2}_{X_0}}{2}-\int_{\Omega}\gamma(x)|u(x)|^{\theta}dx+\kappa,\nonumber
\end{eqnarray*}
where
$$
  \kappa:=(\|\gamma\|_\infty r^\theta+C_1(r^q+1))\meas\,(\Omega).
  $$

  \indent Hence, inequality \eqref{15} is proved.\par
  \noindent At this point, observe that the functional $\|\cdot\|_\gamma:X_0\rightarrow \erre$ defined by
  $$
  \|u\|_\gamma:=\left(\int_\Omega \gamma(x)|u(x)|^\theta dx\right)^{1/\theta},
  $$
  is a norm in $X_0$.\par
  Since in $W$ the norms $\|\cdot\|_{X_0}$ and $\|\cdot\|_\gamma$ are equivalent ($W$ is finite dimensional), there exists a positive constant $\kappa_{W}$ such that
  $$
  \|u\|_{X_0}\leq \kappa_{W}\|u\|_\gamma,
  $$
  for every $u\in X_0$.\par

  \indent Consequently, we have that
  $$
  J_K(u)\leq \displaystyle{\frac{\kappa_{W}^2}{2}\|u\|^{2}_\gamma}-\|u\|_\gamma^{\theta}+\kappa,
  $$
  for every $u\in W$.\par
  Hence
  $$
  \displaystyle{\frac{\kappa_{W}^2}{2}\|u\|^{2}_\gamma}-\|u\|_\gamma^{\theta}+\kappa\geq 0,
  $$
  for every
  $$
  u\in\{u\in W:J_K(u)\geq 0\}.
  $$

  \indent Since $\theta>2$ we conclude that the above set is bounded in $X_0$.
\end{proof}

\textbf{Proof of Theorem \ref{Esistenza} concluded}. With the above notations, we can write
$$
X_0=E_1\oplus E_2,
$$
where $E_1$, given by ${\rm Span}\{e_j:j<k_0\}$, is the orthogonal complement of $E_2$. Thanks to Lemma \ref{lemma2} and Propositions \ref{proposition1}-\ref{proposition2}, Theorem \ref{ZMPT} implies that the functional $J_K$ possesses an unbounded sequence of critical value $\{J_K(u_k)\}_{k\in\mathbb{N}}$, where $u_k$ is a weak solution of \eqref{Nostro}. Since $J'_K(u_k)(u_k)=0$,
\begin{equation}\label{u1}
\int_Q|u_k(x)-u_k(y)|^2K(x-y)dxdy=\int_\Omega f(x,u_k(x))dx,
\end{equation}
and it follows that
\begin{equation}\label{u2}
J_K(u_k)=\int_{\Omega}\left[\frac{1}{2}f(x,u_k(x))u_k(x)-F(x,u_k(x))\right]dx\rightarrow +\infty,
\end{equation}
as $k\rightarrow \infty$. Hence by \eqref{u1}-\eqref{u2} and $(\rm h_2)$, the sequence $\{u_k\}_{k\in\enne}$ must be unbounded in $X_0$ and in $L^{\infty}(\Omega)$. The proof is complete.\par
\smallskip
 In conclusion, we present a simple and direct application of Theorem \ref{Esistenza}.
\begin{example}\label{esempio0}\rm
 Let $s\in (0,1)$, $n>2s$ and $\Omega$ be an open bounded set of
$\RR^n$ with Lipschitz boundary and consider the following nonlocal problem:
\begin{equation} \tag{$P_{K}$} \label{Nostro2}
\left\{\begin{array}{ll}
-\mathcal L_K u=u^3+u & \mbox{in}\,\,\,\Omega\\
u=0 & \mbox{in}\,\,\,\erre^n\setminus\Omega,
\end{array}\right.
\end{equation}
where $K:\RR^n\setminus\{0\}\rightarrow(0,+\infty)$ is a map
satisfying $(\rm k_1)$--$(\rm k_3)$.\par
\noindent Owing to Theorem \ref{Esistenza}, problem \eqref{Nostro2} admits infinitely many weak solutions.
\end{example}

\medskip
 \indent \textbf{Acknowledgements.} A part of this paper was done while the author was visiting the Department
of Mathematics ``Guido Castelnuovo" at the University ``La Sapienza" in March 2013. He expresses his gratitude to Professor Filomena Pacella for the kind invitation and for many stimulating conversations on the subject of the paper.

\end{document}